%% file: template.tex
\documentclass[12pt]{extarticle}
\input{prefix}%

\BibLatexMode{%
   \bibliography{template}
}

\begin{document}

\title{ Smirnov Decomposition of a Horizontal Vector Charge in the Heisenberg Group}

\author{
Zhengyao Huang, Wilhelm Klingenberg
}
\date{\today}

\maketitle
 
\begin{abstract}
\noindent
A divergence-free horizontal vector current in Heisenberg space may be viewed as an element of the dual space of horizontal test vector fields. By applying a horizontal Liouville theorem in this setting, the flow lines of such a vector field generate a family of horizontal curves and an associated measure on this collection. In this paper, we provide a direct proof of the Smirnov decomposition for a Federer-Fleming current within the horizontal distribution.
\end{abstract}
\makeatletter
\def\oversortoftilde#1{\mathop{\vbox{\m@th\ialign{##\crcr\noalign{\kern3\p@}%
      \sortoftildefill\crcr\noalign{\kern3\p@\nointerlineskip}%
      $\hfil\displaystyle{#1}\hfil$\crcr}}}\limits}

\def\sortoftildefill{$\m@th \setbox\z@\hbox{$\braceld$}%
  \braceld\leaders\vrule \@height\ht\z@ \@depth\z@\hfill\braceru$}

\makeatother

\tableofcontents
\clearpage 
\section{Introduction}
Let $T$ be an $\mathbb{R}^n$-valued, countably additive set function defined on the Borel $\sigma$-algebra $\mathscr{B}_n$ of $\mathbb{R}^n$:
$$
T(E)=\left(T_1(E), \ldots, T_n(E)\right), \quad E \in \mathscr{B}_n,
$$
where each $T_j$ is a finite real measure on $\mathscr{B}_n$ (a scalar charge). We refer to $T$ as a \emph{vector charge}. The space of all such vector charges is endowed with the total variation norm
$$
\operatorname{var}(T):=\sup\left\{ \sum_j\left|T\left(E_j\right)\right| : E_j \text{ is a countable Borel partition of }\mathbb{R}^n\right\}.
$$
By the Riesz--Kakutani theorem, a vector charge with finite total variation may be identified with a Federer--Fleming $1$-current of finite mass.

A celebrated result due to S. Smirnov \cite{Smirnov1993} provides a decomposition of vector charges with finite total variation into an integral of elementary curve charges. More precisely, there exists a finite measure $\nu$ on the space $K_\ell$ of Lipschitz curves of some length $\ell$ such that
\[
\mu = \int_{K_\ell} [\gamma]\,d\nu(\gamma).
\]
Our main result in Theorem 4,1  reads as follows.
\begin{theorem}\label{thm: MT}
 For every $\ell > 0$ and every horizontal charge $\mu = (\alpha_1q,\dots,\alpha_mq)$ where $\alpha = (\alpha_1,\dots,\alpha_m) : \mathbb{H}^n \to H \mathbb{H}^n$ is a $q$-measurable horizontal vector field with $q$ a Radon measure with $\operatorname{div}_H \mu = 0$, there is a finite positive measure $\nu$ on $K_\ell$ with
$$\mu = \int_{K_\ell} [\gamma] d\nu(\gamma), \quad \left\|\mu\right\|_{\mathcal{M}} = \int_{K_\ell} \left\|[\gamma]\right\|d\nu(\gamma)$$
and $\nu$-almost every $\gamma \in K_\ell$ has values in $\operatorname{supp}(\mu)$ and satisfies $\|[\gamma]\| = L(\gamma) = \ell$.
\end{theorem}
This decomposition holds both for solenoidal charges (whose distributional divergence vanishes) and for charges whose distributional divergence is a finite measure. Subsequently, Stepanov, Paolini, and Georgiev \cite{PAOLINI20123358,PAOLINI20131269} generalized this decomposition to the setting of acyclic metric currents and metric cycles. It is also worth noting that analogous decompositions arise naturally in the study of $BV$-functions \cite{Alberti_1993}, where a measure $\mu$ on a metric space $X$ can be represented as an integral over path fragments $\operatorname{Frag}(X)$:
\[
\mu=\int_{\operatorname{Frag}(X)} \nu_\gamma \, dP(\gamma).
\]

Smirnov's decomposition theory also plays a significant role in the study of the flat chain conjecture. In their foundational work, Ambrosio and Kirchheim \cite{ambrosio2000currents} observed that every metric current $T \in \mathbf{M}_k(\mathbb{R}^n)$ corresponds to a Federer--Fleming current $\widetilde{T} \in \mathscr{M}_k(\mathbb{R}^n)$, and that this correspondence restricts to a linear isomorphism between the spaces of normal currents, $\mathbf{N}_k(\mathbb{R}^n)$ and $\mathscr{N}_k(\mathbb{R}^n)$.

The full scope of this correspondence, however, remains an open question. The space of flat $k$-chains $\mathscr{F}_k(\mathbb{R}^n)$ is defined as the completion of the normal Federer--Fleming $k$-currents under the flat norm:
\[
\mathbf{F}(T) = \inf\left\{\mathbf{M}(R) + \mathbf{M}(S): R +\partial S = T\right\}.
\]
While Ambrosio and Kirchheim showed that every flat chain $F \in \mathscr{F}_k(\mathbb{R}^d)$ induces a metric current $T \in \mathbf{M}_k(\mathbb{R}^n)$, it is unclear whether the converse holds. In its most concise form, the $k$-flat chain conjecture ($k$-FCC) asserts the natural identification:
\[
\mathscr{F}_k(\mathbb{R}^d) \cong \mathbf{M}_k(\mathbb{R}^d).
\]
Recently, Rabasa and Bouchitté \cite[Theorem~A]{arroyorabasa2025structuralpropertiesonedimensionalmetric} extended the metric $1$-current decomposition from Lipschitz curves to the class of $\operatorname{SBV}(I,X)$ curves, establishing a density characterization of the $1$-FCC.
\noindent
In this paper, we study the Smirnov decomposition for horizontal vector charges in the Heisenberg group $\mathbb{H}^n$. We provide a direct proof adapted from the arguments given by Arenas and Wengenroth \cite{rodríguezarenas2024smirnovdecompositionsvectorfields}.

Note that Ambrosio--Kirchheim metric currents of finite mass in the Heisenberg group coincide with Federer--Fleming horizontal currents of finite mass \cite{franchi2025currentsheisenberggroups} and any horizontal vector charge of finite total variation corresponds to a metric current. Since the decomposition is already established at the metric current level by Stepanov and Paolini's work \cite{PAOLINI20123358}, it can be shown to descend to a classical Federer-Fleming current. The direct proof presented here offers a self-contained approach and illuminates structural properties of the decomposition specific to the Heisenberg setting.
\clearpage 
\section{Horizontal Vector Measures}
Let $\mathbb{H}^n$ be the Heisenberg group, identified with $\mathbb{R}^{2n+1}$ and written in coordinates
$$
(x,y,z)=(x_1,\dots,x_n,y_1,\dots,y_n,z).
$$
The standard horizontal left-invariant vector fields are
$$
X_i=\partial_{x_i}-\frac{1}{2}y_i\,\partial_z,
\qquad
Y_i=\partial_{y_i}+\frac{1}{2}x_i\,\partial_z,
\qquad i=1,\dots,n.
$$
Accordingly, the horizontal bundle is the rank-$2n$ subbundle
$$
H_x \mathbb{H}^n=\operatorname{span}\left\{X_1(x),\dots,X_n(x),Y_1(x),\dots,Y_n(x)\right\}
\quad \forall x \in \mathbb{H}^n.
$$
A subriemannian structure is defined on $\mathbb{H}^n$ once one endows each fiber $H_x \mathbb{H}^n$ of the horizontal bundle $H \mathbb{H}^n$ with a scalar product $\langle\cdot, \cdot\rangle_x$; its associated norm is denoted as $|\cdot|_x$. When clear from the context, we will drop the subscript $x$, simply writing $\langle\cdot, \cdot\rangle$ and $|\cdot|$.

Measurable sections of the horizontal bundle $H \mathbb{H}^n$ are called horizontal sections, or simply horizontal vector fields, and elements of $H_x \mathbb{H}^n$ are called horizontal vectors.

Accordingly, any horizontal vector field $\Phi: \mathbb{H}^n \to H \mathbb{H}^n$ can be written as
$$
\Phi=\sum_{i=1}^n \phi_i X_i+\sum_{i=1}^n \phi_{n+i}Y_i.
$$
Equivalently, we identify $\Phi$ with its coefficient vector
$$
\Phi=\left(\phi_1,\dots,\phi_n,\phi_{n+1},\dots,\phi_{2n}\right).
$$
\begin{definition}
(Vector measure). Let $\mathscr{B}(\mathbb{H}^n)$ be the Borel $\sigma$-algebra on the Heisenberg group $\mathbb{H}^n$. We define a horizontal vector measure on $\mathbb{H}^n$ as an $2n$-tuple $\mu = (\mu_1,\dots,\mu_{2n}) = \sum_{i=1}^n \mu_i X_i + \sum_{i=1}^n \mu_{i+n} Y_i$, where each component $\mu_k : \mathscr{B}(\mathbb{H}^n) \to \mathbb{R}$ is a real-valued ( i.e. signed) Radon measure. Every vector measure $\mu$ acts on test fields $\Phi = (\phi_1,\dots,\phi_{2n}) \in C_c^{\infty} (\mathbb{H}^n, H \mathbb{H}^n) =: \mathcal{D}(\mathbb{H}^n, H \mathbb{H}^n)$ as
$$ \langle \mu, \Phi \rangle = \sum_{k=1}^{2n} \int_{\mathbb{H}^n} \phi_k d\mu_k.$$
We denote $\mathcal{M}(\mathbb{H}^n,H\mathbb{H}^n)$ the space of all vector measures on $\mathbb{H}^n$.
\end{definition}

Since $\mathcal{D}(\mathbb{H}^n, H \mathbb{H}^n)$ is dense in $C_0(\mathbb{H}^n, H \mathbb{H}^n)$, if we take $T \in \mathcal{D}(\mathbb{H}^n, H \mathbb{H}^n)^*$ such that $$\sup \left\{T(\Phi): \Phi \in \mathcal{D}(\mathbb{H}^n, H \mathbb{H}^n),\|\Phi\|_{\infty} \leq 1\right\}<+\infty,$$ we can extend uniquely $T$ to an element of $C_0(\mathbb{H}^n, H \mathbb{H}^n)^*$. Hence, any $T$ turns out to be associated with a vector measure $\mu \in \mathcal{M}(\mathbb{H}^n, H \mathbb{H}^n)$. We henceforth set
$$
\|\mu\|_{\mathcal{M}}:=\sup \left\{T(\Phi): \Phi \in \mathcal{D}(\mathbb{H}^n, H \mathbb{H}^n),\|\Phi\|_{\infty} \leq 1\right\}=\|T\|_{C_0^*}
$$
This is the total variation of $\mu$, which is also denoted as $\operatorname{var}(\mu)$.

From the polar decomposition, $\mu = \alpha |\mu|$, where $\alpha : \mathbb{H}^n \to \mathbb{R}^{2n}$, $|\alpha| = 1$, and $|\mu|$ is the variation measure of $\mu$. We define the horizontal divergence $\operatorname{div}_H \mu$ as the distribution acting on $f \in \mathcal{D}(\mathbb{H}^n,\mathbb{R})$ by
$$
\langle \operatorname{div}_H \mu, f \rangle = - \langle \mu, \nabla_H f \rangle.
$$
Substituting $\mu = \alpha |\mu|$, we obtain
$$
\langle \operatorname{div}_H \mu, f \rangle = - \int_{\mathbb{H}^n} \langle \nabla_H f(x), \alpha(x) \rangle_x \, d|\mu|(x).
$$
Writing
$$
\alpha=\sum_{i=1}^n \alpha_i X_i+\sum_{i=1}^n \alpha_{n+i}Y_i
$$
with respect to the frame $\{X_1,\dots,X_n,Y_1,\dots,Y_n\}$, this becomes
$$
\langle \operatorname{div}_H \mu, f \rangle
=-\sum_{i=1}^n \int_{\mathbb{H}^n} \alpha_i(X_i f) \, d|\mu|
-\sum_{i=1}^n \int_{\mathbb{H}^n} \alpha_{n+i}(Y_i f) \, d|\mu|.
$$

It is worth noting that a horizontal vector measure with finite total variation, i.e., a vector charge, can be identified with a horizontal Federer-Fleming 1-current \cite{franchi2025currentsheisenberggroups}. For completeness, we provide the details of this identification.
\begin{definition}
(Horizontal and Vertical Differential Forms, 2.3, \cite{franchi2025currentsheisenberggroups}). Let us define
$$
\mathfrak{h}_1 = \operatorname{span}\left\{X_1,\dots, X_n, Y_1,\dots, Y_n\right\},
\qquad
\mathfrak{h}_2 = \operatorname{span}\{Z\},
$$
and set $\mathfrak{h} = \mathfrak{h}_1 \oplus \mathfrak{h}_2$. The dual space of $\mathfrak{h}$ is denoted by $\bigwedge^1 \mathfrak{h}$. The basis of $\bigwedge^1 \mathfrak{h}$ dual to $\left\{X_1, \ldots, X_n, Y_1, \ldots, Y_n, Z\right\}$ is the family of covectors $\left\{d x_1, \ldots, d x_n, d y_1, \ldots, d y_n, \theta\right\}$, where
$$
\theta:=d z-\frac{1}{2} \sum_{j=1}^n\left(x_j d y_j-y_j d x_j\right)
$$
is the contact form on $\mathbb{H}^n$. We denote by $\langle\cdot, \cdot\rangle$ the inner product on $\bigwedge^1 \mathfrak{h}$ for which $$\left(d x_1, \ldots, d x_n, d y_1, \ldots, d y_n, \theta\right)$$ is an orthonormal basis.

We set
$$
\omega_i:=d x_i, \quad \omega_{i+n}:=d y_i, \quad \text{and} \quad \omega_{2 n+1}:=\theta, \quad \text{for } i=1, \ldots, n.
$$
We put $\bigwedge^0 \mathfrak{h}=\mathbb{R}$ and, for $1 \leq h \leq 2 n+1$,
$$
\bigwedge^h \mathfrak{h}:=\operatorname{span}\left\{\omega_{i_1} \wedge \cdots \wedge \omega_{i_h}: 1 \leq i_1<\cdots<i_h \leq 2 n+1\right\}.
$$
Similarly,
$$
\bigwedge^h \mathfrak{h}_1:=\operatorname{span}\left\{\omega_{i_1} \wedge \cdots \wedge \omega_{i_h}: 1 \leq i_1<\cdots<i_h \leq 2 n\right\}.
$$
Throughout this paper, elements of $\bigwedge^h \mathfrak{h}$ are identified with the left-invariant differential forms of degree $h$ on $\mathbb{H}^n$. In degree $k$, one has the orthogonal decomposition
$$
\bigwedge^k \mathfrak{h}=\bigwedge^k \mathfrak{h}_1 \oplus \theta \wedge \bigwedge^{k-1} \mathfrak{h}_1,
$$
where forms in the first summand are called horizontal and forms in the second summand are called vertical. In particular, a differential form is vertical if and only if it is divisible by the contact form $\theta$.
\end{definition}
\begin{definition}
(Definition 3.2, \cite{franchi2025currentsheisenberggroups}). A Federer-Fleming current $T$ is horizontal if $T$ and $\partial T$ vanish on vertical forms. Equivalently, if
$$
T\llcorner\theta=0, \quad T\llcorner d\theta=0.
$$
\end{definition}
We claim that there is a one-to-one correspondence between vector charges and horizontal Federer-Fleming $1$-currents. We begin with a horizontal Federer-Fleming $1$-current $T$. Let $\varphi \in \mathcal{D}(\mathbb{H}^n)$. By the Riesz-Kakutani representation theorem, there exists a unique Radon measure $\mu_T$ such that
$$
(T\llcorner \theta)(\varphi) = T(\varphi \theta) = 0 = \int_{\mathbb{H}^n} \langle \varphi\theta, \vec{T} \rangle \, d\mu_T,
$$
where $\vec{T}$ is a $\mu_T$-measurable vector field on $\mathbb{H}^n$ with $|\vec{T}| = 1$ for $\mu_T$-a.e. Hence $\theta(x) \perp \vec{T}(x)$ for $\mu_T$-a.e. $x \in \mathbb{H}^n$, and therefore $\vec{T}$ must be horizontal. We may thus define the horizontal vector measure $\mu := \vec{T}\,\mu_T$, which has finite mass.

Conversely, let $\mu$ be a horizontal vector measure with $\|\mu\|< \infty$, written as
$$
\mu = \sum_{i=1}^n \mu_i X_i + \sum_{i=1}^n \mu_{i+n} Y_i.
$$
Define a current $T_\mu$ by
$$
T_\mu(\omega) := \int_{\mathbb{H}^n} \langle \omega, d\mu \rangle = \sum_{i=1}^{2n} \int_{\mathbb{H}^n} \omega_i \, d\mu_i.
$$
Then $T_\mu$ is a Federer-Fleming current. Moreover,
$$
(T_\mu \llcorner \theta)(\varphi) = T_\mu(\varphi \theta) = \int_{\mathbb{H}^n} \langle \varphi \theta, d\mu \rangle = 0,
\qquad \forall \varphi \in \mathcal{D}(\mathbb{H}^n).
$$
Therefore $T_\mu$ has finite mass and is horizontal.
\section{The Collection of Horizontal Curves}
\begin{definition}
(Horizontal curve). An absolutely continuous curve $\gamma: [a, b] \to \mathbb{H}^n$ is called a horizontal curve if its tangent vector lies strictly within the horizontal bundle almost everywhere: $$\dot{\gamma}(t) \in H_{\gamma(t)}\mathbb{H}^n \quad \text{for a.e. } t \in [a, b].$$
The length of a horizontal curve is defined as
$$L(\gamma) := \int_a^b \langle \dot{\gamma}(t),\dot{\gamma}(t)\rangle^{1/2} \; dt.$$
\end{definition}

The Heisenberg-Carathéodory distance, $d_{CC}(x, y)$, between any two points $x, y \in \mathbb{H}^n$ is defined as the infimum of the lengths of all horizontal curves connecting them:$$d_{CC}(x, y) = \inf \left\{ L(\gamma) \mid \gamma: [a, b] \to \mathbb{H}^n \text{ is horizontal}, \, \gamma(a) = x, \, \gamma(b) = y \right\}.$$A curve $\gamma : [a,b] \to \mathbb{H}^n$ is said to be Lipschitz (specifically 1-Lipschitz) with respect to this metric if it satisfies:$$d_{CC}\left(\gamma(s),\gamma(t)\right)\leq |s-t| \quad \text{for all } s,t \in [a,b].$$According to Proposition 1.1 in Hajłasz and Tyson \cite[p. 3]{Zimmerman},  Lipschitz curves are necessarily horizontal. A 1-Lipschitz curve yields a curve charge $[\gamma]$ which acts on horizontal test fields as a curve integral
$$\langle [\gamma],\Phi\rangle = \int_{[a,b]}\langle \Phi(\gamma)(t),\dot{\gamma}(t)\rangle dt.$$
We can consider $[\gamma]$ as a (horizontal) vector measure on $\mathbb{H}^n$, and the components of $[\gamma]$ can be considered as the image measures of $\dot{\gamma}_k \mathcal{L}^1$ under $\gamma$, i.e. for $A \in \mathscr{B}(\mathbb{H}^n)$:
$$\gamma_\# (\dot{\gamma}_k \mathcal{L}^1)(A) = \int_{\{t: \gamma(t) \in A\}} \dot{\gamma}_k(t) d\mathcal{L}^1(t).$$
For open set $E \subset \mathbb{H}^n$,$$|[\gamma]|(E) \leq \int_{\{t: \gamma(t) \in E\}} |\dot{\gamma}|dt,\quad \operatorname{var}([\gamma])\leq L(\gamma).$$
This inequality is an equality if and only if
$$|[\gamma]|(A) = \int_{\{t: \gamma(t) \in A\}} |\dot{\gamma}(t)|dt,\quad \forall A \in \mathscr{B}(\mathbb{H}^n).$$
We can also calculate the horizontal divergence of $[\gamma]$. For all $\psi \in \mathcal{D}(\mathbb{H}^n,\mathbb{R})$,
$$\operatorname{div}_H([\gamma]) (\psi) = - \int_{[a,b]} \langle \nabla_H \psi(\gamma(t)),\dot{\gamma}(t)\rangle dt = -\int_a^b(d\psi)_{\gamma(t)}(\dot{\gamma}(t)) = \delta_{\gamma(a)}-\delta_{\gamma(b)}(\psi).$$

For a fixed $\ell > 0$ we denote
$$
K_{\ell}=\left\{\gamma:[0, \ell] \rightarrow \mathbb{H}^n: \gamma \text { is Lipschitz with }|\dot{\gamma}| \leq 1\text { $\mathcal{L}^1$-a.e.}\right\} .
$$
We define a homogeneous norm $\|\cdot\|_{\mathbb{H}^n}$ to be the distance to the identity element $e$:
$$\|x\|_{\mathbb{H}^n} = d_{C C}(x, e)$$
For paths in $C([0,\ell], \mathbb{H}^n)$, we define the uniform norm of a curve $\gamma$ as the supremum of the group norm along the path:$$||\gamma||_{\infty} = \sup_{t \in [0,\ell]} ||\gamma(t)||_{\mathbb{H}^n}.$$
The associated metric of this uniform norm is given by
$$d_\infty(\gamma_1,\gamma_2) = \sup_{t\in [0,\ell]} \|\gamma_1(t)^{-1} \cdot \gamma_2(t)\|_{\mathbb{H}^n}$$
By Definition 2.10 in Carfagnini \cite[p. 702]{Carfagnini2021OnTS}, $\left(C([0,\ell], \mathbb{H}^n),d_\infty\right)$ is a complete metric space. Furthermore, we remark that the Carnot-Carathéodory distance $d_{C C}$ and the distance induced by any homogeneous norm $\|\cdot\|_{\mathbb{H}^n}$ are bilipschitz equivalent. That is, there exist constants $c, C>0$ such that
$$
c d_{C C}(p, q) \leq\left\|p^{-1} \cdot q\right\|_{\mathbb{H}^n} \leq C d_{C C}(p, q) \quad \forall p, q \in \mathbb{H}^n.
$$
\begin{lemma}
The set of all 1-Lipschitz curves with length less or equal to $\ell$ in $\mathbb{H}^n$, denoted as $K_\ell$, is $\sigma$-compact subset of  $\left(C([0,\ell], \mathbb{H}^n),d_\infty\right)$.
\end{lemma}
\begin{proof}
Define $F_m = \left\{\gamma \in K_\ell: \|\gamma(0)\|_{\mathbb{H}^n} \leq m\right\}$. We need to show that $F_m$ is compact. By the Arzelà-Ascoli theorem, we need to show $F_m$ is equicontinuous in $\mathbb{H}^n$, closed and pointwise compact. Notice that the equicontinuity follows directly from the 1-Lipschitz condition. Let $\varepsilon > 0$ and choose $\delta = \varepsilon$. For any $s, t \in [0, \ell]$ with $|s-t| < \delta$, the 1-Lipschitz condition guarantees $d_{CC}(\gamma(s),\gamma(t)) \le |s-t| < \varepsilon$. Since $d_{\mathbb{H}^n}$ is equivalent to $d_{CC}$, this implies $F_m$ is uniformly equicontinuous in $(\mathbb{H}^n, d_{\mathbb{H}^n})$.

We show that $F_{m,x} = \left\{\gamma(x): \gamma \in F_m\right\}$ is bounded. Since $\gamma$ is 1-Lipschitz in $d_{CC}$, the distance traveled is at most $\ell$. For any $\gamma \in F_m$:$$d_{CC}(e, \gamma(x)) \le d_{CC}(e, \gamma(0)) + d_{CC}(\gamma(0), \gamma(x))\leq m + \ell.$$
Therefore $\|\gamma\|_{\mathbb{H}^n}$ is uniformly bounded. Since $F_{m,x}$ is bounded and $\mathbb{H}^n$ is a proper metric space, we conclude that $F_{m,x}$ is relatively compact in $\mathbb{H}^n$.

Finally, we show that $F_m$ is closed in $\left(C([0,\ell], \mathbb{H}^n),d_\infty\right)$ with respect to the $d_\infty$-topology. Let us consider a sequence $\{\gamma_n\}_{n\in \mathbb{N}} \subset F_m$ with a uniform limit $\gamma \in \left(C([0,\ell], \mathbb{H}^n),d_\infty\right)$ such that $$\lim_{n \to \infty} d_\infty(\gamma_n, \gamma)= 0.$$
We want to show $\gamma \in F_m$. But the uniform convergence implies the pointwise convergence, hence
$\|\gamma(0)\|_{\mathbb{H}^n} = \lim_{n \to \infty} \|\gamma_n(0)\|_{\mathbb{H}^n} \leq m$. Furthermore, since $d_{\mathbb{H}^n}$ and $d_{CC}$ are equivalent, 
$$\lim_{n \to \infty} d_{CC}(\gamma_n(t),\gamma(t)) = 0.$$
Hence
$$d_{CC}(\gamma(s),\gamma(t))=\lim_{n \to \infty} d_{CC}(\gamma_n(s),\gamma_n(t))\leq |s-t|.$$
We conclude that $\gamma \in F_m$. Therefore $F_m$ is closed. 

Because $F_m$ is compact for every $m \in \mathbb{N}$, and $K_\ell = \bigcup_{m=1}^\infty F_m$, it follows that $K_\ell$ is a countable union of compact sets. Therefore, $K_\ell$ is $\sigma$-compact.
\end{proof}

To compactify $K_\ell$, we consider the one-point compactification of the Heisenberg group $\mathbb{H}^n$:$$\hat{\mathbb{H}}^n = \mathbb{H}^n \cup \{\infty\}.$$This compactification can be explicitly constructed as a CR generalization of the stereographic projection, mapping the Heisenberg group—viewed as the boundary of a Siegel domain—to the unit sphere $S^{2n+1}$ \cite[p. 637]{balogh_uniformly_2012}. We define the Siegel domain $D \subset \mathbb{C}^{n+1}$ as:$$D = \left\{ (\zeta, \zeta_0) \in \mathbb{C}^n \times \mathbb{C} \mid \operatorname{Im} \zeta_0 - |\zeta|^2 > 0 \right\},$$where $|\zeta|^2 = \zeta \cdot \bar{\zeta}$ is the standard Euclidean norm in $\mathbb{C}^n$.The Heisenberg group $\mathbb{H}^n$ is the space of elements $(z, t) \in \mathbb{C}^n \times \mathbb{R}$ equipped with the group operation:$$(z, t)(z', t') = \left(z + z', t + t' + 2 \operatorname{Im}(z \cdot \bar{z}')\right).$$
The Heisenberg group $\mathbb{H}^n$ operates transitively on $D$ by
$$
\begin{aligned}
\mathbb{H}^n \times D & \rightarrow D \\
\left((z, t),\left(\zeta, \zeta_0\right)\right) & \mapsto\left(\zeta+z, \zeta_0+t+2 \mathrm{i} \zeta \cdot \bar{z}+\mathrm{i}|z|^2\right)
\end{aligned}
$$
This group action extends continuously to the boundary $\partial D$. By evaluating the action at the origin $(0,0) \in \partial D$, we obtain a natural identification between the Heisenberg group and $\partial D$. This identification is given by the embedding $\iota: \mathbb{H}^n \to \partial D$:
$$\iota(z, t) = (z, t + i|z|^2).$$
Let $B \subset \mathbb{C}^{n+1}$ be the unit ball, with its boundary sphere of dimension $2n+1$ defined as:$$\partial B = \left\{ (w, w_0) \in \mathbb{C}^n \times \mathbb{C} \mid |w|^2 + |w_0|^2 = 1 \right\}.$$The Cayley transformation $C : B \to D$, defined by$$C(w, w_0) = \left( \frac{iw}{1+w_0}, i \frac{1-w_0}{1+w_0} \right),$$is a holomorphic bijection. It extends to a continuous bijection between the boundaries, mapping $\partial B \setminus \{(0,-1)\} \to \partial D$. Its inverse, $C^{-1}: D \to B$, is given by:$$C^{-1}(\zeta, \zeta_0) = \left( \frac{2\zeta}{\zeta_0 + i}, \frac{i - \zeta_0}{\zeta_0 + i} \right).$$Finally, we compose the inverse Cayley transform with our boundary identification to define the map $F(z,t) = C^{-1}(\iota(z,t))$. This explicitly maps the Heisenberg group $\mathbb{H}^n$ to the punctured sphere $\partial B \setminus \{(0,-1)\}$:$$F(z, t) = \left( \frac{2z}{t + i(1+|z|^2)}, \frac{-t + i(1-|z|^2)}{t + i(1+|z|^2)} \right).$$By defining $F(\infty) = (0,-1) \in \partial B$, we smoothly extend this map from $\mathbb{H}^n$ to the one-point compactification $\hat{\mathbb{H}}^n$, thereby identifying the compactified Heisenberg group to the entire unit sphere $\partial B$.

Define $I : C([0,\ell], \mathbb{H}^n) \to C([0,\ell],\hat{\mathbb{H}}^n)$ by $I(\gamma) = F \circ \gamma$ where $C([0,\ell],\hat{\mathbb{H}}^n)$ is also given with a uniform topology. Let us consider a sequence $\{\gamma_k\}_{k \in \mathbb{N}} \subset K_\ell$ with a uniform limit $\gamma$. We analyze this limit by considering the behavior of the starting points $\gamma_k(0)$. Suppose the sequence of starting points $\gamma_k(0)$ does not converge to infinity, then by taking limits of both sides:
$$|\gamma_k(t)|\leq |\gamma_k(0)| + \ell \Rightarrow |\gamma(t)|\leq M + \ell.$$
Consequently, the limiting curve after inversion, $\hat{\gamma} = F \circ \gamma$, only takes finite values and lies entirely away from the pole $(0,-1)$. Therefore, $\hat{\gamma} \in I(K_\ell)$. Now suppose the sequence of starting points escapes to infinity, meaning $|\gamma_k(0)| \to \infty$ as $k \to \infty$, then
$$|\gamma_k(t)| \geq |\gamma_k(0)| - \ell.$$
The limiting curve after inversion $\hat{\gamma} = F \circ \gamma$ is therefore the constant curve at the pole $(0,-1)$, meaning $\hat{\gamma} \in \overline{I(K_\ell)} \setminus I(K_\ell)$. By the Arzelà–Ascoli theorem, $\hat{K_\ell} \subset C([0,\ell], S^{2n+1})$ is a compact subset. Hence we can apply the Riesz-Kakutani representation theorem on $C \left(\hat{K}_\ell\right)$ to find a regular finite Borel measure on $\hat{K_\ell}$.
\section{Liouville theorem on the Heisenberg group} 
Consider $\mathbb{H}^1=(\mathbb{R}^3,d_{CC})$, with horizontal bundle
$$H\mathbb{H}^1=\operatorname{span}\{X,Y\},$$
and let $Z=[X,Y]$, where
$$
X(x, y, z) :=\frac{\partial}{\partial x}-\frac{y}{2} \frac{\partial}{\partial z},\quad
Y(x, y, z) :=\frac{\partial}{\partial y}+\frac{x}{2} \frac{\partial}{\partial z}
$$
If $\phi \in C^1_H\left(\mathbb{H}^1, H \mathbb{H}^1\right)$, then
$$\phi(x) = \phi_1(x) X(x) + \phi_2(x) Y(x)$$
where $\phi_1,\phi_2: \mathbb{H}^1 \to \mathbb{R}$ satisfy
$$X(\phi_1), Y(\phi_1), X(\phi_2), Y(\phi_2)$$
being continuous functions.
However, this does not imply any regularity for the Lie derivative along $Z=[X,Y]$, since
$$[X,Y]\phi_1 = X\circ Y(\phi_1) - Y\circ X(\phi_1)$$
involves second-order mixed derivatives. Therefore, $C_H^1$ regularity alone is not sufficient for an ODE uniqueness statement based on the Lie bracket structure. To apply the Carath\'eodory theorem \cite[p. 53]{Agrachev_Barilari_Boscain_2019} and obtain uniqueness for the associated flow (Cauchy problem), one needs continuity of all relevant directional derivatives, including those generated by commutators. A natural sufficient assumption is $\phi \in C_H^2\left(\mathbb{H}^1,H\mathbb{H}^1\right)$.
\begin{remark}
It is worth noting that, in general, a Lipschitz horizontal vector field on a homogeneous group does not guarantee uniqueness for the associated Cauchy problem \cite{Magnani2016OnLV}. The main issue is that intrinsic (horizontal) Lipschitz regularity does not imply Euclidean Lipschitz regularity of the full vector field. As shown by Magnani and Trevisan, uniqueness can be recovered only under additional structural assumptions (for instance, in the equilibrium-point setting or when the field is constrained to an involutive submodule of the horizontal bundle) \cite{Magnani2016OnLV}.
\end{remark}
We further assume that
$$|\phi(x)|_x \leq c\bigl(1 + \|x\|_{\mathbb{H}^1}\bigr)$$
for some constant $c > 0$ and all $x \in \mathbb{H}^1$, where the homogeneous norm is defined by
$$
\left\|\left(x_1, x_2, x_3\right)\right\|_{\mathbb{H}^1}:=\left(\left(\left|x_1\right|^2+\left|x_2\right|^2\right)^2+\left|x_3\right|^2\right)^{1 / 4}
$$
It is straightforward to show that $\phi$ is a complete vector field. Consider a smooth curve $\gamma: I \to \mathbb{H}^1$ solving the Cauchy problem
$$\dot{\gamma}(t) = \phi(\gamma(t)).$$
Write
$$\gamma(t) = (h(t),z(t)) = (x_1(t),x_2(t),x_3(t)),$$
with
$$\dot{h}(t) = \bigl(\phi_1(\gamma(t)),\phi_2(\gamma(t))\bigr), \qquad \dot{z}(t) = \frac{1}{2}\,\omega\bigl(h(t),\dot{h}(t)\bigr),$$
where $\omega$ is the standard symplectic form. Assume that $I$ is a finite interval. We have bounds:
$$|\dot{h}(t)| = |\phi(\gamma(t))| \leq c\bigl(1 + \|\gamma(t)\|_{\mathbb{H}^1}\bigr),$$
and
$$|\dot{z}(t)| = \frac{1}{2}\bigl|\omega\bigl(h(t),\dot{h}(t)\bigr)\bigr|\leq \frac{1}{2}|h(t)|\,|\dot{h}(t)|\leq \frac{c}{2} \left(\|\gamma(t)\|^2_{\mathbb{H}^1} + \|\gamma(t)\|_{\mathbb{H}^1}\right).$$
Define $\mathbf{S}(x) = \|x\|_{\mathbb{H}^1}^4$.
Then
\begin{align*}
\left|\frac{d}{dt}\mathbf{S}(\gamma(t))\right| &\leq \left|4|h(t)|^2 \langle h(t),\dot{h}(t)\rangle + 2 z(t) \dot{z}(t)\right|\\
&\leq 4c \|\gamma(t)\|^3_{\mathbb{H}^1}\bigl(1 + \|\gamma(t)\|_{\mathbb{H}^1}\bigr) + c \|\gamma(t)\|_{\mathbb{H}^1}\left(\|\gamma(t)\|_{\mathbb{H}^1}^2 + \|\gamma(t)\|_{\mathbb{H}^1}\right)\\
&\leq K\bigl(\|\gamma(t)\|^4_{\mathbb{H}^1}+1\bigr)
\end{align*}
for some constant $K > 0$ (depending on the finite time interval $I$). Using Grönwall's inequality, we obtain
$$\mathbf{S}(t)\leq \bigl(\mathbf{S}(0) + 1\bigr)e^{Kt}-1.$$
By the escape lemma \cite[Lemma 9.19]{Johnlee}, we conclude that $\phi$ is complete.
Therefore, this setup has a globally defined flow $u \in C^1(\mathbb{R} \times \mathbb{H}^n)$, where $u(\cdot, x)$ is the unique solution to the initial value problem $\partial_t u(t,x) = \phi(u(t,x))$.
\begin{proposition}
(Liouville theorem). Let $\phi \in C_H^2(\mathbb{H}^n, H\mathbb{H}^n)$ be a horizontal vector field with sublinear growth. Let $\varrho \in \mathcal{M}(\mathbb{H}^n)$ be a a locally finite Borel measure on $\mathbb{H}^n$, and assume that $\operatorname{div}_H(\phi\ \varrho)=0$. Then $\varrho$ is invariant under the flow $u$ generated by $\phi$; namely, $\varrho_t := u(t,\cdot)_\#\varrho$ is independent of $t \in \mathbb{R}$.
\end{proposition}
\begin{proof}
The proof will be based on \cite{rodríguezarenas2024smirnovdecompositionsvectorfields}. For $\varphi \in \mathcal{D}(\mathbb{H}^n, \mathbb{R})$, it is enough to show that$$\int_{\mathbb{H}^n} \varphi d\varrho_t = \int_{\mathbb{H}^n} \varphi(u(t,x))d\varrho(x)$$is independent of $t$. Define $\psi(t,x) = \varphi(u(t,x))$. Because $\phi$ has sublinear growth, the flow $u$ is a globally defined diffeomorphism, and the flow equation yields $\psi(t,u(-t,x)) = \varphi(x)$. Taking the total derivative with respect to $t$ gives$$0 = \frac{d}{dt} \big(\psi(t,u(-t,x))\big) = (\partial_t\psi)(t,u(-t,x)) - \langle \nabla \psi(t,u(-t,x)),\phi(u(-t,x))\rangle.$$Since the gradient decomposes into a horizontal and a vertical part, $\nabla f = \nabla_H f + \nabla_V f$, and $\phi$ is a horizontal vector field, evaluating the above at $y = u(-t,x)$ yields the transport equation$$\partial_t \psi(t,y) = \langle \nabla_H \psi(t,y),\phi(y)\rangle,$$where $\langle \cdot, \cdot \rangle$ is the inner product on the horizontal bundle. We fix a compact interval $J = [-m,m]$. For $t \in J$, we have$$\operatorname{supp}(\psi(t,\cdot)) \subseteq \{x \in \mathbb{H}^n: u(t,x) \in \operatorname{supp}(\varphi)\} = u(-t,\operatorname{supp}(\varphi)) \subseteq u(J \times \operatorname{supp}(\varphi)).$$Notice that $K = u(J \times \operatorname{supp}(\varphi))$ is compact. Since $\partial_t \psi(t,\cdot)$ is continuous, $|\partial_t \psi(t,\cdot)|$ is uniformly bounded on $K$. We can thus differentiate with respect to $t \in J$ under the integral sign:$$\partial_t \int \psi(t,x) d\varrho(x) = \int \partial_t \psi(t,x) d\varrho(x) = \int \langle \nabla_H \psi(t,x),\phi(x)\rangle d\varrho(x) = - \langle\operatorname{div}_H(\phi \varrho),\psi(t,\cdot)\rangle.$$By hypothesis, $\operatorname{div}_H(\phi \varrho) = 0$ in the sense of distributions, so the derivative is zero, proving the claim.
\end{proof}
\section{Smirnov decomposition of Horizontal Charges}
We now state and prove the Smirnov decomposition of a horizontal charge with vanishing horizontal divergence.
\begin{theorem}\label{thm: divergence free}
(Smirnov's theorem for $\operatorname{div}\mu = 0$). For every $\ell > 0$ and every horizontal charge $\mu = (\alpha_1q,\dots,\alpha_mq)$ where $\alpha = (\alpha_1,\dots,\alpha_m) : \mathbb{H}^n \to H \mathbb{H}^n$ is a $q$-measurable horizontal vector field with $q$ a Radon measure with $\operatorname{div}_H \mu = 0$, there is a finite positive measure $\nu$ on $K_\ell$ with
$$\mu = \int_{K_\ell} [\gamma] d\nu(\gamma), \quad \left\|\mu\right\|_{\mathcal{M}} = \int_{K_\ell} \left\|[\gamma]\right\|d\nu(\gamma)$$
and $\nu$-almost every $\gamma \in K_\ell$ has values in $\operatorname{supp}(\mu)$ and satisfies $\|[\gamma]\| = L(\gamma) = \ell$.
\end{theorem}
\begin{definition}
(Definition 2.9, \cite{MONTEFALCONE2007453}). Let $\nabla$ be the left-invariant Levi-Civita connection on $\mathbb{H}^n$ associated with
\[
g : T\mathbb{H}^n \times T\mathbb{H}^n \to \mathbb{R},
\qquad
T\mathbb{H}^n = \operatorname{Span}_{\mathbb{R}}\{X_1,\dots,X_n,Y_1,\dots,Y_n,Z\}.
\]
Moreover, for $V, W \in \mathbf{C}^{\infty}(\mathbb{H}^n, H\mathbb{H}^n)$, we define
\[
\nabla_V^H W := p_H(\nabla_V W).
\]
The operator $\nabla^H$ is called the horizontal connection.
\end{definition}
It is straightforward to verify that $\nabla^H$ is flat, compatible with the sub-Riemannian metric
\[
g_H : H\mathbb{H}^n \times H\mathbb{H}^n \to \mathbb{R},
\]
and torsion-free. These properties follow directly from the definition of $\nabla^H$ and from the corresponding properties of the Levi-Civita connection $\nabla$. Therefore, the global left-invariant frame $\{X_1,\dots,X_n,Y_1,\dots,Y_n\}$ allows us to define uniformly continuous horizontal vector fields on $\mathbb{H}^n$ unambiguously.
\begin{definition}
(Uniform Continuity on Heisenberg Group). We say that a horizontal vector field $V : \mathbb{H}^n \to H\mathbb{H}^n$ is uniformly continuous if there exists a modulus of continuity $\omega$ such that for all $p, q \in \mathbb{H}^n$, the Euclidean distance between their coefficient vectors is bounded by a modulus of convergence $\omega_V$:
$$\left(\sum_{k=1}^{2n}\left|V^k(p)-V^k(q)\right|^2\right)^{1/2}\leq \omega_V\left(d_{CC}(p,q)\right).$$
\end{definition}
\begin{remark}
To ensure that Smirnov's integral is well defined, it is enough to prove the Borel measurability of the map
$$
\gamma \mapsto \langle [\gamma],\varphi \rangle,
\qquad
\varphi \in \mathcal{D}(\mathbb{H}^n, H\mathbb{H}^n).
$$
For clarity, we present the argument in $\mathbb{H}^1$ (the general case is identical, with $2n$ horizontal components). Write
$$
\varphi=\varphi_1 X+\varphi_2 Y,
\qquad
\dot\gamma=\dot x\,X_{\gamma}+\dot y\,Y_{\gamma}.
$$
For $m\in\mathbb{N}$, let $t_k=k\ell/m$ ($k=0,\dots,m$), and define
$$
S_m(\gamma):=\sum_{k=1}^{m}\left\langle
\varphi(\gamma(t_k)),
\begin{pmatrix}
x(t_k)-x(t_{k-1})\\
y(t_k)-y(t_{k-1})
\end{pmatrix}
\right\rangle_{\gamma(t_k)} = \sum_{k=1}^m \langle \varphi(\gamma(t_k)),\Delta \gamma_k \rangle.
$$
Since $\gamma$ is absolutely continuous and horizontal (in particular, 1-Lipschitz in $d_{CC}$), the Riemann sums $S_m(\gamma)$ converge to
$$
\int_0^{\ell}\langle \varphi(\gamma(t)),\dot\gamma(t)\rangle\,dt
=\langle[\gamma],\varphi\rangle
$$
as $m\to\infty$.
For fixed $m$, the map $\gamma\mapsto S_m(\gamma)$ is continuous on $K_\ell$ with respect to $\|\cdot\|_\infty$. Indeed, $\varphi:\mathbb{H}^1\to H\mathbb{H}^1$ is bounded and uniformly continuous. Let $M:=\sup_{p\in\mathbb{H}^1}\|\varphi(p)\|$ and let $\omega_\varphi$ be a modulus of continuity for $\varphi$. Then
$$
\begin{aligned}
|S_m(\gamma)-S_m(\eta)|
&=\left|\sum_{k=1}^{m}\langle\varphi(\gamma_k),\Delta\gamma_k\rangle_{\gamma_k}
-\sum_{k=1}^{m}\langle\varphi(\eta_k),\Delta\eta_k\rangle_{\eta_k}\right|\\
&\leq \sum_{k=1}^m\left|\left\langle \mathbf{\Pi}\varphi(\gamma_k),\mathbf{\Pi}\Delta\gamma_k-\Delta\eta_k\right\rangle_{\eta_k}\right|
+\sum_{k=1}^m\left|\left\langle \mathbf{\Pi}\varphi(\gamma_k)-\varphi(\eta_k),\Delta\eta_k\right\rangle_{\eta_k}\right|\\
&\leq M\sum_{k=1}^m\|\mathbf{\Pi}\Delta\gamma_k-\Delta\eta_k\|_{\eta_k}
+\sum_{k=1}^m\omega_\varphi\!\left(d_{CC}(\gamma_k,\eta_k)\right)\,\|\Delta\eta_k\|_{\eta_k}.
\end{aligned}
$$
Here $\mathbf{\Pi}$ denotes left translation from $\gamma_k$ to $\eta_k$, which preserves horizontal components. Using
$$
\|\mathbf{\Pi}\Delta\gamma_k-\Delta\eta_k\|_{\eta_k}
\leq 2\,d_\infty(\gamma,\eta),
\qquad
 d_{CC}(\gamma_k,\eta_k)\le d_\infty(\gamma,\eta),
$$
we obtain
$$
|S_m(\gamma)-S_m(\eta)|
\leq (2mM)\,d_\infty(\gamma,\eta)
+\omega_\varphi\!\bigl(d_\infty(\gamma,\eta)\bigr)
\sum_{k=1}^{m}\|\Delta\eta_k\|.
$$
Here $L_m(\eta):=\sum_{k=1}^{m}\|\Delta\eta_k\|$ is finite for fixed $\eta$ and $m$. Hence, as $d_\infty(\gamma,\eta)\to0$, the right-hand side tends to $0$, proving continuity of $S_m$. Since
$$
\langle [\gamma],\varphi\rangle=\lim_{m\to\infty}S_m(\gamma)
$$
is a pointwise limit of continuous functions, it is Borel measurable.
\end{remark}
We now begin the proof of Theorem \ref{thm: divergence free},
namely Smirnov's theorem on the Heisenberg group.
\begin{proof}
Let $\mu =(\mu_1,\dots,\mu_m)$ be a divergence-free charge satisfying $\operatorname{div}_H \mu = 0$, and let $\ell > 0$. Let $J: \mathbb{H}^n \to \mathbb{R}$ be a mollifier, i.e.
$J \in C^\infty_c(\mathbb{H}^n)$, $\operatorname{supp}(J) \subset B_1(0)$, and $\int_{\mathbb{H}^n}J(x)\, dx = 1$. Define the convolution $\tilde{\mu} = \mu * J = (\mu_1 * J,\dots, \mu_m * J)$, which is a smooth horizontal vector field satisfying $|\mu* J|\leq |\mu|* J$ on $\mathbb{H}^n$. Define
$$
\phi = \frac{\mu * J}{|\mu|* J}.
$$
Let $\varrho = (|\mu|*J)\cdot h$, where $h$ is the Haar measure on $\mathbb{H}^n$. This is a finite measure, since it is the convolution of $|\mu|$ with the probability measure $J\cdot h \in \mathcal{P}(\mathbb{H}^n)$. Furthermore,
$$\operatorname{div}_H(\phi \cdot \varrho) = \operatorname{div}_H((\mu * J) h)= (\operatorname{div}_H(\mu)*J)\cdot h = 0.$$
Therefore we can apply Liouville's theorem to the flow generated by $\phi$ and the measure $\varrho$.

We now show that $\tilde{\mu} =\mu * J$ decomposes into the curves $\gamma_x :[0,\ell] \to \mathbb{H}^n$, $t \mapsto u(t,x)$. Indeed, for $\varphi \in \mathcal{D}(\mathbb{H}^n, H\mathbb{H}^n)$, the invariance of $\varrho$ under $u(t,\cdot)$ yields
\begin{align*}
\int_{\mathbb{H}^n} [\gamma_x](\varphi)d\varrho(x) &= \int_{\mathbb{H}^n} \int_0^\ell \langle \varphi(u(t,x)),\partial_t u(t,x)\rangle dt d\varrho(x)\\
&=\int_0^\ell \int_{\mathbb{H}^n} \langle \varphi(u(t,x)),\phi(u(t,x))\rangle d\varrho(x) dt\\
&=\int_0^\ell \int_{\mathbb{H}^n} \langle \varphi, \phi\rangle d\varrho dt = \ell \int_{\mathbb{H}^n} \langle \varphi,\phi\rangle(|\mu|*J)dh\\
&= \ell \langle \tilde{\mu}, \varphi \rangle.
\end{align*}
Since $|\phi|\leq 1$, the above identity shows that $\gamma_x \in K_\ell$. Define a map $\Gamma: \mathbb{H}^n \rightarrow K_{\ell}$ by $x \mapsto \gamma_x$. Let $g: K_{\ell} \rightarrow \mathbb{R}$ be a $\Gamma_\# \varrho$-measurable function. Then
$$
\int_{K_{\ell}} g(\gamma) d\left[\frac{\Gamma_{\#} \varrho}{\ell}\right](\gamma)=\int_{\mathbb{H}^n} g\left(\gamma_x\right) \frac{1}{\ell} d \varrho(x) .
$$
Taking $g(\cdot)=\langle[\cdot], \varphi\rangle$, we conclude that $\nu\left(K_{\ell}\right)=\operatorname{var}(\mu * J) / \ell$ and
$$
\tilde{\mu}=\int_{K_{\ell}}[\gamma] d \nu(\gamma).
$$
For $\varepsilon > 0$ we set $J_\varepsilon(x) = \frac{1}{\varepsilon^4}J(\delta_{1/\varepsilon}(x))$, where $\delta_{1/\varepsilon}$ is the dilation acting on the group element $x$:
$$\delta_{\lambda}(x_1,\dots,x_{2n},x_{2n+1}) = (\lambda x_1,\dots,\lambda x_{2n}, \lambda^2 x_{2n+1}).$$
Then $\langle \mu * J_\varepsilon, \varphi\rangle \to \langle \mu, \varphi\rangle$ for all $\varphi \in \mathcal{D}(\mathbb{H}^n, H\mathbb{H}^n)$, and we obtain measures $\{\nu_\varepsilon\}$ on $K_\ell$ with $\nu_\varepsilon(K_\ell)= \operatorname{var}(\mu * J_\varepsilon)/\ell$ such that
$$\mu * J_\varepsilon = \int_{K_\ell}[\gamma]d\nu_\varepsilon(\gamma).$$
We can consider $\nu_\varepsilon$ as measures on $\hat{K}_\ell$ with $\nu_\varepsilon(\{\infty\}) = 0$. Since $\{\nu_\varepsilon\}$ is uniformly bounded by $\operatorname{var}(\mu)/\ell$, by Banach-Alaglou's theorem, there exists a subsequence $\nu_{\varepsilon_k} \to \nu$ in the weak-$*$ topology:
$$\int f d\nu_{\varepsilon_k} \to \int f d\nu,\quad \forall f \in C(\hat{K}_\ell).$$

For all $\varphi \in \mathcal{D}(\mathbb{H}^n, H\mathbb{H}^n)$, we want to show
$$\langle \mu,\varphi \rangle = \int_{K_\ell} [\gamma](\varphi) d\nu(\gamma).$$
Define the map $f: \hat{K}_\ell \to \mathbb{R}$ by $f(\gamma) = [\gamma](\varphi)$. It suffices to prove that $f$ is continuous. Note that if $\gamma$ is the constant curve at $\infty$, then $f(\gamma) = 0$ because $\varphi$ has compact support in $\mathbb{H}^n$. To prove continuity, let $\{\gamma_j\} \subset \hat{K}_\ell$ be a sequence converging uniformly to $\gamma$. We must show that $[\gamma_j](\varphi) \to [\gamma](\varphi)$. By Urysohn's subsequence principle, the sequence $\{[\gamma_j](\varphi)\}$ converges to $[\gamma](\varphi)$ if and only if every subsequence contains a further subsequence converging to $[\gamma](\varphi)$. Since every subsequence of $\{\gamma_j\}$ still converges uniformly to $\gamma$, it is enough to prove the following: for every sequence $\{\gamma_j\}$ converging uniformly to $\gamma$, there exists a subsequence $\gamma_{j_k}$ such that $[\gamma_{j_k}](\varphi) \to [\gamma](\varphi)$.

Suppose first that $\gamma_j$ converges to the constant curve at $\infty$. Then $\varphi \circ \gamma_j = 0$ for all sufficiently large $j$, so $[\gamma_j](\varphi)=0$ eventually.

Otherwise, the sequence $\{\gamma_j\}$ is uniformly bounded. Since $|\dot{\gamma}_j|\leq 1$ and the velocity can be identified with a vector in $\mathbb{R}^{2n}$,
$$\dot{\gamma}_j(t) = (\dot{\gamma}_1(t),\dots,\dot{\gamma}_{2n}(t)) \in \mathbb{R}^{2n},$$
we have that $\{\dot{\gamma}_j\}$ is bounded in the Hilbert space $L^2([0,\ell],\mathbb{R}^{2n})$. Hence, by weak compactness of closed balls in Hilbert spaces, we can extract a subsequence (still denoted $\dot{\gamma}_{j(i)}$) such that $\dot{\gamma}_{j(i)} \rightharpoonup g$ in $L^2([0,\ell],\mathbb{R}^{2n})$.

Since each $\gamma_j$ is absolutely continuous and $d_\infty(\gamma_j,\gamma) \to 0$, we obtain for every $t \in [0,\ell]$,
$$\gamma(t) = \lim_{i\to\infty}\gamma_{j(i)}(0) + \lim_{i\to\infty}\int_0^t \dot{\gamma}_{j(i)}(s)\,ds = \gamma(0) + \int_0^t g(s)\,ds,$$
which implies $\dot{\gamma}(t)=g(t)$ for almost every $t\in[0,\ell]$. Therefore,
$$
\begin{aligned}
[\gamma_{j(i)}](\varphi)-[\gamma](\varphi)
&=\sum_{k=1}^{2n}\int_0^\ell\bigl(\varphi_k(\gamma_{j(i)}(t)) - \varphi_k(\gamma(t))\bigr)\dot{\gamma}_{j(i)}^k(t)\,dt\\
&
+\sum_{k=1}^{2n}\int_0^\ell\varphi_k(\gamma(t))\bigl(\dot{\gamma}_{j(i)}^k(t)-\dot{\gamma}^k(t)\bigr)\,dt.
\end{aligned}
$$
The first term tends to $0$ by uniform convergence of $\gamma_{j(i)}$ to $\gamma$ and boundedness of $\dot{\gamma}_{j(i)}$, while the second term tends to $0$ by weak convergence of $\dot{\gamma}_{j(i)}$ to $\dot{\gamma}$ in $L^2([0,\ell],\mathbb{R}^{2n})$.

Next we show that $|\mu| = \int |[\gamma]|d\nu(\gamma)$ and that $\nu$-almost all curves have length $\ell$. So far we only know that their length is bounded above by $\ell$. For an open set $E \subseteq \mathbb{H}^n$ and $\varphi \in \mathcal{D}(E,H\mathbb{H}^n)$ with $\varphi_i \in C^\infty_c(E)$ and $|\varphi|\leq 1$, we have
$$\langle \mu,\varphi\rangle = \int_{K_\ell} [\gamma](\varphi)d\nu(\gamma) \leq \int_{K_\ell} |[\gamma]|(E) d\nu(\gamma),$$
and hence $|\mu|(E)\leq \int_{K_\ell}|[\gamma]|(E) d\nu(\gamma)$. For $E = \mathbb{H}^n$, we have
$$\operatorname{var}(\mu) \leq \int_{K_\ell} \operatorname{var}([\gamma])d\nu(\gamma) \leq \int_{K_\ell}L(\gamma)d\nu(\gamma) \leq \ell \nu(K_\ell).$$
Since we have $\operatorname{var}(\mu) = \ell \nu(K_\ell)$, all these are equalities. This implies that $\operatorname{var}([\gamma]) = L(\gamma) = \ell$ for $\nu$-almost all $\gamma \in K_\ell$ and $|\mu| = \int |[\gamma]|d\nu(\gamma)$ by the regularity of Borel measures.

Finally, we show that $\nu$-almost every curve takes values in $\operatorname{spt}\mu$. First recall that
$$|[\gamma]|(E) \leq \int_{\{t: \gamma(t) \in E\}}|\dot{\gamma}(t)|dt.$$
For $\gamma \in K_\ell$ with $\operatorname{var}([\gamma]) = L(\gamma) = \ell$, we thus get $|\dot{\gamma}| = 1$ for almost all $t \in [0,\ell]$, together with equality in the previous line. For $E = \operatorname{spt}(\mu)^c$ we obtain
$$0 = |\mu|(E) = \int_{K_\ell}\mathcal{L}^1(\{t \in [0,\ell]: \gamma(t) \in E\})d\nu(\gamma)$$
which implies that the open sets $\{t \in [0,\ell] : \gamma(t) \in E\}$ are empty for $\nu$-almost all $\gamma$.
\end{proof}

\begin{theorem}
(Smirnov's theorem for $\operatorname{div}\mu$ being a signed measure) Let $\mu$ be a horizontal vector charge on $\mathbb{H}^n$ whose weak horizontal divergence $\operatorname{div}_H(\mu)$ is a finite signed measure. Then there exist a constant $\ell > 0$ and a finite positive measure $\nu$ supported on $K_\ell$, the space of horizontal curves in $\mathbb{H}^n$, such that
$$
\mu = \int_{K_\ell} [\gamma] \, d\nu(\gamma).
$$
In particular, the representing curves have length at most $\ell$. Moreover, this decomposition satisfies the total variation identity
$$
|\mu| = \int_{K_\ell} |[\gamma]| \, d\nu(\gamma),
$$
and for $\nu$-almost every $\gamma \in K_\ell$, the curve $\gamma$ takes values in $\operatorname{supp}(\mu)$ and satisfies $\operatorname{var}([\gamma]) = L(\gamma)$.
\end{theorem}
\begin{proof}
Let $\mu = (\mu_1,\dots,\mu_{2n})$ be a horizontal vector charge on $\mathbb{H}^n$ such that $\operatorname{div}_H(\mu)$ is a signed measure. We define an $(n+1)$-dimensional vector charge by embedding $\mathbb{H}^n \times \mathbb{R}$ into $\mathbb{H}^{n+1}$. Notice that an element of $\mathbb{H}^n$ can be written as $p = (\mathbf{x},\mathbf{y},z) \in \mathbb{R}^n \times \mathbb{R}^n \times \mathbb{R}$, while an element of $\mathbb{H}^{n+1}$ can be written as
$$
(\mathbf{x}, x_{n+1}, \mathbf{y}, y_{n+1}, z).
$$
The horizontal bundle $H\mathbb{H}^{n+1}$ is spanned by $\{X_1, \dots, X_n, X_{n+1}, Y_1, \dots, Y_n, Y_{n+1}\}$, where
$$
X_i = \partial_{x_i} - \frac{1}{2} y_i \, \partial_z
\quad \text{and} \quad
Y_i = \partial_{y_i} + \frac{1}{2} x_i \, \partial_z,
\qquad i = 1, \dots, n+1.
$$
Define the embedding $E : \mathbb{H}^n \times \mathbb{R} \hookrightarrow \mathbb{H}^{n+1}$ by setting $x_{n+1} = t$ and $y_{n+1} = 0$:
$$
E((\mathbf{x},\mathbf{y},z), t) = (\mathbf{x}, t, \mathbf{y}, 0, z).
$$
Restricted to the image $E(\mathbb{H}^n \times \mathbb{R})$, the $(n+1)$st horizontal vector field simplifies to
$$
X_{n+1}|_{E(\mathbb{H}^n \times \mathbb{R})} = \partial_{x_{n+1}} = \partial_t.
$$
Thus this left-invariant horizontal direction behaves like a flat Euclidean direction, orthogonal to the other frame fields and decoupled from the central coordinate $z$. Moreover, this embedding preserves horizontality of curves. Indeed, let $\gamma(s) = (\mathbf{x}(s),\mathbf{y}(s),z(s))$ be a horizontal curve in $\mathbb{H}^n$, and let $t(s)$ be a trajectory in $\mathbb{R}$. Then the lifted curve $\hat{\gamma}(s) = E(\gamma(s), t(s))$ in $\mathbb{H}^{n+1}$ satisfies
$$
\dot{z} = \frac{1}{2} \sum_{i=1}^n (x_i \dot{y}_i - y_i \dot{x}_i) + \frac{1}{2} \bigl( x_{n+1} \dot{y}_{n+1} - y_{n+1} \dot{x}_{n+1} \bigr)
= \frac{1}{2} \sum_{i=1}^n (x_i \dot{y}_i - y_i \dot{x}_i),
$$
because $y_{n+1} = 0$ and $\dot{y}_{n+1} = 0$, so the second term vanishes identically. We now define a lifted horizontal vector charge $\mu^+$ on $\mathbb{H}^{n+1}$, supported on $E(\mathbb{H}^n \times \mathbb{R})$. In the frame $\{X_1,\dots, X_n, X_{n+1}, Y_1,\dots, Y_n, Y_{n+1}\}$, we set
$$
\mu^+ = \left(\mu \otimes (\delta_0 - \delta_\ell), -\operatorname{div}_H(\mu) \otimes \mathcal{H}|_{[0,\ell]}, \mathbf{0}_{Y_{n+1}} \right),
$$
where the first component lies in $H\mathbb{H}^n$, the second is strictly in the direction of $X_{n+1}$, and the $Y_{n+1}$ component is zero. The natural projection $\pi : H \mathbb{H}^{n+1} \to H \mathbb{H}^n$ restricts $\mu^+$ back down; for any Borel set $A \in \mathcal{B}(\mathbb{H}^n)$, we have $\pi \circ\mu^+ (A \times \{0\}) = \mu(A)$.

Let $\varphi \in \mathcal{D}(E(\mathbb{H}^n \times \mathbb{R}))$ be a scalar test function. To distinguish the horizontal gradients on $\mathbb{H}^n$ and $\mathbb{H}^{n+1}$, we keep the notation $\nabla_H$ for the horizontal gradient on $\mathbb{H}^n$ and write $\nabla_H^{(n+1)}$ for the horizontal gradient on $\mathbb{H}^{n+1}$. Along $E(\mathbb{H}^n \times \mathbb{R})$, the additional component is the partial derivative with respect to $t$, and thus
$$
\nabla_H^{(n+1)}\varphi = (\nabla_H\varphi,\partial_t\varphi, 0).
$$
By duality, the weak divergence is defined as $\left\langle\operatorname{div}^{(n+1)}_H(\mu^+), \varphi\rangle = -\langle\mu^+, \nabla_H^{(n+1)}\varphi\right\rangle$. Therefore
$$\begin{aligned}
\left\langle \operatorname{div}^{(n+1)}_H(\mu^+), \varphi \right\rangle
&= -\left\langle \mu \otimes (\delta_0 - \delta_{\ell}), \nabla_H \varphi \right\rangle
- \left\langle -\left.\operatorname{div}_H(\mu) \otimes \mathcal{H}^1\right|_{[0, \ell]}, \partial_t \varphi \right\rangle \\
&= -\left\langle \mu, \nabla_H \varphi(\cdot, 0) - \nabla_H \varphi(\cdot, \ell) \right\rangle
+ \left\langle \operatorname{div}_H(\mu), \int_{[0, \ell]} \partial_t \varphi(\cdot, t) \, dt \right\rangle \\
&= \left\langle \operatorname{div}_H(\mu), \varphi(\cdot, 0) - \varphi(\cdot, \ell) \right\rangle
+ \left\langle \operatorname{div}_H(\mu), \varphi(\cdot, \ell) - \varphi(\cdot, 0) \right\rangle \\
&= 0.
\end{aligned}$$
Theorem \ref{thm: divergence free} yields a finite positive measure $\nu^+$ on 
$$
K_{\ell}^+=\left\{\gamma^+:[0, \ell] \rightarrow \mathbb{H}^{n+1}: \gamma^+ \text { is Lipschitz with }|\dot{\gamma}| \leq 1 \text{ for } \mathcal{H}^1 \text {-a.e. }\right\}
$$
with
$$\mu^+ = \int_{K_\ell^+} [\gamma^+]d\nu^+(\gamma^+),\quad |\mu^+| = \int_{K_\ell^+}|[\gamma^+]|d\nu^+(\gamma^+).$$
We first analyze the behavior of the curves in $S = \mathbb{H}^n \times (0,\ell)$. We have
$$\mu^+\llcorner S = (0,-\operatorname{div}_H(\mu)\otimes \mathcal{H}^1\llcorner(0,\ell)),\quad |\mu^+\llcorner S| = |\operatorname{div}_H(\mu)|\otimes \mathcal{H}^1\llcorner (0,\ell)$$
and $\nu^+$-almost all $\gamma^+ \in K_\ell^+$ satisfy $\operatorname{var}([\gamma]) = L(\gamma) = \ell$. The polar decomposition is $\mu^+ \llcorner S = f\cdot|\operatorname{div}_H(\mu)|\otimes \mathcal{H}^1 \llcorner (0,\ell)$, where $f(x,t) = (0,g(x))$ with $-\operatorname{div}_H(\mu) = g\cdot |\operatorname{div}_H(\mu)|$ where $|g(x)| =1$ for $|\operatorname{div}_H(\mu)|$-a.e. $x \in \mathbb{H}^n$. We set
$$E_\pm = \left\{x \in \operatorname{supp}(\operatorname{div}_H(\mu)): g(x) = \pm 1\right\},\quad \operatorname{supp}(\operatorname{div}_H(\mu)) = E_+ \cup E_-.$$
We claim for $\nu^+$-almost all curves $\gamma^+$ and $\mathcal{H}^1$-almost all parameter times $s$ that $\gamma^+(s) \in E_\pm \times (0,\ell)$ implies $\dot{\gamma}^+_{n+1}(s) = \pm 1$ (in the direction of $X_{n+1}$).

To see this, let $-d\operatorname{div}_H(\mu) = g d|\operatorname{div}_H(\mu)|$ be the Radon-Nikodym density, and define the sets $E_\pm = \{x \in \mathbb{H}^n : g(x) = \pm 1\}$. Let $\chi_+$ be the indicator function of the set $E_+ \times (0,\ell)$. Using the parameter $s$ for the curves, we have the following sequence of inequalities:
\begin{align*}
&\nu^+\otimes \mathcal{H}^1\llcorner(0,\ell) \left(\{(\gamma^+,t): \gamma^+(t) \in E_+ \times (0,\ell)\}\right) = \int_{K_\ell}\int_0^\ell \chi_+(\gamma^+(t))d\nu^+\otimes d\mathcal{H}^1(\gamma^+,t)\\
&\geq \int_{K_\ell^+}\int_0^\ell (\chi_+ \circ \gamma^+)(t)\dot{\gamma}^+_{n+1}(t)d\mathcal{H}^1(t)d\nu^+(\gamma^+) = \int_{K_\ell^+} \langle [\gamma^+],(0,\dots,0,\chi_+) \rangle d\nu^+(\gamma^+)\\
&= \langle \mu^+,(0,\dots,0,\chi_+)\rangle = \int_{\mathbb{H}^{n+1}}\chi_+ d(-\operatorname{div}_H(\mu) \otimes \mathcal{H}^1\llcorner(0,\ell))\\
&=\int_{\mathbb{H}^{n+1}}\chi_+ gd\left(|\operatorname{div}_H(\mu)|\otimes \mathcal{H}^1\llcorner(0,\ell)\right) = |\mu^+|(E_+ \times (0,\ell))\\
&=\int_{K_\ell^+}|[\gamma^+]|(E_+ \times (0,\ell))d\nu^+(\gamma^+) = \int_{K_\ell^+} \mathcal{H}^1\left((\gamma^+)^{-1}(E_+ \times (0,\ell))\right)d\nu^+(\gamma^+)\\
&=\nu^+ \otimes \mathcal{H}^1\llcorner(0,\ell)\left(\{(\gamma^+,t): \gamma^+(t) \in E_+ \times (0,\ell)\}\right).
\end{align*}
We conclude that the above inequalities are in fact equalities. It follows that$$\chi_+(\gamma^+(s))\dot{\gamma}^+_{n+1}(s) = \chi_+(\gamma^+(s)),\quad \text{for } \mathcal{H}^1 \text{-a.e. } s.$$
Since curves in $K_\ell^+$ are absolutely continuous with $|\dot{\gamma}^+| \leq 1$, forcing $\dot{\gamma}^+_{n+1} = 1$ simultaneously forces all other components of the velocity to zero. We obtain that $\nu^+$-almost all curves are strictly vertical when they pass through the vertical regions $S= \operatorname{supp}(|\operatorname{div}_H(\mu)|) \times (0,\ell)$. In particular, the only non-zero component of $\dot{\gamma}^+$ is $X_{n+1}$.

We next define
$$D = \left\{\gamma^+ \in K_\ell^+ : \gamma^+(0) \in (E_-\times (0,\ell))\cup(\mathbb{H}^n \times (0,\ell)) \right\}.$$
Let $\alpha(\gamma^+) = \inf\left\{s : \gamma^+(s) \in \mathbb{H}^n \times \{0\}\right\}$ and $\beta(\gamma^+) = \sup\left\{s : \gamma^+(s) \in \mathbb{H}^n \times \{0\}\right\}$. For $\nu^+$-almost every curve $\gamma^+ \in D$ we then have exactly one contiguous interval on the horizontal plane:$$\{s : \gamma^+(s) \in \mathbb{H}^n \times \{0\}\} = [\alpha(\gamma^+),\beta(\gamma^+)],$$
since $E_+ \cap E_- = \emptyset$ and $\gamma^+$ cannot have more than one sharp turning point between the boundary of the horizontal plane $\mathbb{H}^n \times \{0\}$ and the vertical direction $X_{n+1}$. Furthermore, we have
$$\dot{\gamma}^+(t) = (0,\dots,0, \pm 1),\quad \text{for } t \in [0,\ell] \backslash [\alpha(\gamma),\beta(\gamma)].$$
For $A \in \mathcal{B}(\mathbb{H}^n)$, we have $[\gamma^+](A \times \{0\}) = ([\pi \circ \gamma^+](A),0)$ and since $\nu^+$-almost all curves in $K_\ell^+ \backslash D$ do not meet at the plane $\mathbb{H}^n \times \{0\}$ we obtain
\begin{align*}
(\mu,0)(A\times \{0\}) &= \mu^+\llcorner(\mathbb{H}^n \times \{0\})(A\times\{0\}) = \int_D \langle [\gamma^+], \chi_{A \times \{0\}}\rangle d\nu^+(\gamma^+)\\
&=\int_D (\langle [\pi\circ \gamma^+],\chi_A\rangle,0)d\nu^+(\gamma^+)
\end{align*}
Taking the first $n$-coordinates we thus have
$$\mu = \int_D [\pi \circ \gamma^+]d\nu^+(\gamma).$$
Let us define a reparameterized projection map $\mathbf{P}(\gamma^+) = \pi (\gamma^+(\max\left\{\alpha(\gamma^+), \min\left\{s,\beta(\gamma^+)\right\}\right\}))$.  Then 
$$\mathbf{P}_\#(\nu^+\llcorner D) = \nu.$$
Then $\nu$-almost all curves in $K_{\ell}$ have values in $\operatorname{supp}(\mu)$ and satisfy $\operatorname{var}([\gamma])=L(\gamma)$ because for $\nu^{+}$-almost all curves $\gamma^+$ the projections $\pi \circ \gamma^+$ have this property. Finally, we preserve the total variation $|\mu|=\int_{K_{\ell}}|[\gamma]| d \nu(\gamma)$ due to the corresponding length-conservation property for $\mu^{+}$ and the identity $|\mu|(A)=\left|\mu^{+}\right|(A \times\{0\})$.
\end{proof}
We finally state an application of Theorem 5.5. Its proof follows analogously as the corresponding Theorem 2.3 in [RW24].

\begin{corollary}
(Havin-Smirnov). Let $M \subseteq \mathbb{H}^1$ be a compact set and all rectifiable curves in $M$ are constant. Then for any $f,f_1,\dots,f_m \in C(M)$ and for any $\varepsilon > 0$ there exists $\psi \in \mathcal{D}(\mathbb{H}^1, \mathbb{R})$ such that
$$\sup_{x \in M} |f(x)-\psi(x)| + \sum_{j=1}^m \sup_{x\in M}|f_j - (\nabla_H \psi)_j| < \varepsilon.$$
\end{corollary}
\printbibliography
\noindent\textsc{Zhengyao Huang}, Department of Mathematics, Durham University, UK\\
\textit{Email address}: \texttt{phqv76@durham.ac.uk}\\
\noindent\textsc{Wilhelm Klingenberg}, Department of Mathematics, Durham University, UK\\
\textit{Email address}: \texttt{wilhelm.klingenberg@durham.ac.uk}
\end{document}

%% file: prefix.tex
\def\UseBibLatex{1}

\makeatletter
\def\input@path{{styles/}}
\makeatother

\newcommand{\UsePackage}[1]{%
  \IfFileExists{styles/#1.sty}{%
      \usepackage{styles/#1}%
   }{%
      \IfFileExists{../styles/#1.sty}{%
         \usepackage{../styles/#1}%
      }{%
         \usepackage{#1}%
      }%
   }%
}

\usepackage[T1]{fontenc}
\usepackage{bbold}
\usepackage{lmodern}
\usepackage{textcomp}
\usepackage{amsfonts}
\usepackage{mathtools}
\usepackage{amsmath}%
\usepackage{mathrsfs}
\usepackage{amssymb}%
\usepackage[table]{xcolor}%

\usepackage{todonotes}
\usepackage[in]{fullpage}%

\usepackage[amsmath,thmmarks]{ntheorem}%
\theoremseparator{.}%
\usepackage{stmaryrd}
\usepackage{titlesec}%
\titlelabel{\thetitle. }%
\usepackage{xcolor}%
\usepackage{mleftright}%
\usepackage{xspace}%
\usepackage{graphicx}
\usepackage{hyperref}%

\usepackage{tcolorbox}
\definecolor{myviolet}{rgb}{0.73,0.56,0.64}
\newtcolorbox{mybox}{
    arc=0pt,
    boxrule=0pt,
    colback=myviolet,
    width=1\textwidth,   
    colupper=white,
    fontupper=\bfseries
}
\usepackage{hyperref}%
\hypersetup{%
      unicode,
      breaklinks,%
      colorlinks=true,%
      urlcolor=[rgb]{0.25,0.0,0.0},%
      linkcolor=[rgb]{0.5,0.0,0.0},%
      citecolor=[rgb]{0,0.2,0.445},%
      filecolor=[rgb]{0,0,0.4},
      anchorcolor=[rgb]={0.0,0.1,0.2}%
}
\usepackage[ocgcolorlinks]{ocgx2}

\theoremseparator{.}%

\theoremstyle{plain}%
\newtheorem{theorem}{Theorem}[section]

\newtheorem{lemma}[theorem]{Lemma}

\newtheorem{corollary}[theorem]{Corollary}

\newtheorem{proposition}[theorem]{Proposition}

\theoremstyle{plain}%
\theoremheaderfont{\sf} \theorembodyfont{\upshape}%
\newtheorem*{remark:unnumbered}[theorem]{Remark}%
\newtheorem{remark}[theorem]{Remark}%
\newtheorem{definition}[theorem]{\textbf{Definition}}

\newcommand{\myqedsymbol}{\rule{2mm}{2mm}}
\usepackage{cancel}
\theoremheaderfont{\em}%
\theorembodyfont{\upshape}%
\theoremstyle{nonumberplain}%
\theoremseparator{}%
\theoremsymbol{\myqedsymbol}%
\newtheorem{proof}{Proof:}%

\providecommand{\emphind}[1]{}%
\renewcommand{\emphind}[1]{\emph{#1}\index{#1}}

\definecolor{blue25emph}{rgb}{0, 0, 11}

\providecommand{\emphic}[2]{}
\renewcommand{\emphic}[2]{\textcolor{blue25emph}{%
      \textbf{\emph{#1}}}\index{#2}}

\providecommand{\emphi}[1]{}%
\renewcommand{\emphi}[1]{\emphic{#1}{#1}}

\definecolor{almostblack}{rgb}{0, 0, 0.3}

\providecommand{\emphw}[1]{}%
\renewcommand{\emphw}[1]{{\textcolor{almostblack}{\emph{#1}}}}%

\providecommand{\emphOnly}[1]{}%
\renewcommand{\emphOnly}[1]{\emph{\textcolor{blue25}{\textbf{#1}}}}

\newcommand{\HLink}[2]{\hyperref[#2]{#1~\ref*{#2}}}
\newcommand{\HLinkSuffix}[3]{\hyperref[#2]{#1\ref*{#2}{#3}}}

\providecommand{\deflab}[1]{}
\renewcommand{\deflab}[1]{\label{def:#1}}

\providecommand{\eqlab}[1]{}%
\renewcommand{\eqlab}[1]{\label{equation:#1}}

\newcommand{\remove}[1]{}%

\usepackage[inline]{enumitem}

\newlist{compactenumA}{enumerate}{5}%
\setlist[compactenumA]{topsep=0pt,itemsep=-1ex,partopsep=1ex,parsep=1ex,%
   label=(\Alph*)}%

\newlist{compactenuma}{enumerate}{5}%
\setlist[compactenuma]{topsep=0pt,itemsep=-1ex,partopsep=1ex,parsep=1ex,%
   label=(\alph*)}%

\newlist{compactenumI}{enumerate}{5}%
\setlist[compactenumI]{topsep=0pt,itemsep=-1ex,partopsep=1ex,parsep=1ex,%
   label=(\Roman*)}%

\newlist{compactenumi}{enumerate}{5}%
\setlist[compactenumi]{topsep=0pt,itemsep=-1ex,partopsep=1ex,parsep=1ex,%
   label=(\roman*)}%

\newlist{compactitem}{itemize}{5}%
\setlist[compactitem]{topsep=0pt,itemsep=-1ex,partopsep=1ex,parsep=1ex,%
   label=\ensuremath{\bullet}}%

\providecommand{\BibLatexMode}[1]{}
\providecommand{\BibTexMode}[1]{}

\ifx\UseBibLatex\undefined%
  \renewcommand{\BibLatexMode}[1]{}
  \renewcommand{\BibTexMode}[1]{#1}
\else
  \renewcommand{\BibLatexMode}[1]{#1}
  \renewcommand{\BibTexMode}[1]{}
\fi

\BibLatexMode{%
   \usepackage[bibencoding=utf8,style=alphabetic,backend=biber]{biblatex}%
   \UsePackage{my_biblatex}%
}

\numberwithin{figure}{section}%
\numberwithin{table}{section}%
\numberwithin{equation}{section}%

\usepackage{tikz}
\usepackage{tikz-3dplot}
\bibliography{template}
\usetikzlibrary{arrows.meta, positioning, shapes.geometric}
